\documentclass[10pt]{amsart}
\usepackage{latexsym}
\usepackage{amssymb}
\usepackage{graphicx}
\usepackage{epsf}
\usepackage[all]{xypic}
\usepackage{delarray}
\usepackage{setspace}
\newtheorem{theorem}{Theorem}[section]
\newtheorem{corollary}[theorem]{Corollary}

\newtheorem{proposition}[theorem]{Proposition}
\newtheorem{definition}[theorem]{Definition}
\newtheorem{example}[theorem]{Example}
\newtheorem{remark}[theorem]{Remark}
\newcommand{\Hom}{{\rm Hom}}

\newcommand{\Rat}{{\rm Rat}}

\newcommand{\fp}{{\rm fp}}
\newcommand{\fd}{{\rm fd}}

\definemorphism{dato}\dashed \tip \notip
\definemorphism{Dato}\Dashed \tip \notip

\newcommand{\Ker}{{\rm Ker}\,}
\newcommand{\Coker}{{\rm Coker}\,}
\newcommand{\im}{{\rm Im}\,}

\newcommand{\Mor}{{\rm Mor}\,}

\newcommand{\Ff}{\mathcal{F}}

\newcommand{\Mm}{\mathcal{M}}

\def\NN{{\mathbb N}}

\begin{document}
\sloppy

\title[Symmetry for comodule categories]{EQUIVALENCES OF CATEGORIES, GRUSON-JENSEN DUALITY AND
APPLICATIONS}


\begin{abstract} For coalgebras $C$ over a field, we study when the categories ${}^C\Mm$ of left $C$-comodules and
$\Mm^C$ of right $C$-comodules are symmetric categories, in the sense that there is a duality between the  categories of
finitely presented unitary left $R$-modules and finitely presented unitary left $L$-modules, where $R$ and $L$ are the
functor rings associated to the finitely accessible categories  ${}^C\Mm$ and $\Mm^C$.
\end{abstract}

\author{Septimiu Crivei}

\address{Faculty of Mathematics and Computer Science \\ ``Babe\c s-Bolyai" University \\ Str. Mihail Kog\u alniceanu 1
\\ 400084 Cluj-Napoca, Romania} \email{crivei@math.ubbcluj.ro}

\author{Miodrag Cristian Iovanov}

\address{ 
University of Bucharest, Fac. Matematica \& Informatica,
Str. Academiei 14,
Bucharest 010014,
Romania, \&\\
University of Southern California, 3620 South Vermont Ave. KAP 108\\
Los Angeles, CA 90089, USA}
\email{yovanov@gmail.com, iovanov@usc.edu}

\maketitle

\section{Introduction}

Let $R$ be a ring with identity and let $\mathcal{M}_R$ denote the category of unitary right $R$-modules. 
Consider the functor rings associated to the categories $\mathcal{M}_R$ and $\mathcal{M}_{R^{\rm op}}$, say $S$ and
$A^{\rm op}$ respectively. Then there is a duality between the categories of finitely presented unitary left $S$-modules
and finitely presented unitary right $A$-modules. This is a reformulation of the duality proved by Gruson and Jensen
\cite[Theorem~5.6]{GJ}. The result was extended by Dung and Garc\'\i a \cite[Theorem~2.9]{DG} to the case of unitary
modules over a ring with enough idempotents, and furthermore, by Crivei and Garc\'\i a \cite[Corollary~5.13]{CG} to the
case of unitary and torsionfree modules over an idempotent ring, provided the corresponding module categories are
locally finitely generated. The latter was the first Gruson-Jensen duality established for categories not having enough
projectives.

Gruson-Jensen duality can also be reformulated using the notion of symmetric categories. The idea of considering such a 
concept appeared first in the work of Herzog \cite{Herzog}, according to the account given by Prest in
\cite{Prest}, and was later on used by Dung and Garc\'\i a for finitely accessible categories \cite{DG}, and by Crivei
and Garc\'\i a for exactly definable categories \cite{CG}. In this language, the above  Gruson-Jensen dualities say that
the categories $\mathcal{M}_R$ and $\mathcal{M}_{R^{\rm op}}$ are symmetric to each other whenever they are: categories
of unitary modules over a ring with identity, categories of unitary modules over a ring with enough idempotents, or more
generally, categories of unitary and torsionfree modules over an idempotent ring, provided they are locally finitely
generated. Every finitely accessible Grothendieck category has a symmetric category, but this fails in general to be
finitely accessible Grothendieck; see the example from \cite[p.~3953]{CG}. Therefore, an interesting problem is to find
further examples of finitely accessible Grothendieck categories having their symmetric categories again finitely
accessible Grothendieck.

In the present paper we consider categories of comodules over a coalgebra $C$ over a field. We first recall general
results and give an interpretation of Gruson-Jensen duality in terms of Freyd categories and homotopy categories of
certain categories of chain complexes. The categories of left
$C$-comodules and right $C$-comodules are locally finite Grothendieck \cite{DNR}, and so finitely accessible
Grothendieck. The existence of a left-right setting similar to the case of modules suggests that this would be a good
framework for a Gruson-Jensen duality to hold. Let us first note that the categories of left $C$-comodules and right
$C$-comodules are equivalent to categories of modules over rings with identity if $C$ is a finite dimensional coalgebra,
and more generally, they are equivalent to categories of modules over rings with enough idempotents if $C$ is a left and
right semiperfect coalgebra \cite{DNR}. So the categories of left $C$-comodules and of right $C$-comodules are
symmetric, or equivalently, the Gruson-Jensen duality takes place in these cases, because it can be reduced to
the aforementioned module-theoretic contexts. On the other hand, we shall construct several examples of coalgebras with
very good finitary properties for which the Gruson-Jensen duality does not hold. 

Although a coalgebra has very good built-in finiteness properties, which might at first glance lead one to believe that 
such a duality would be in place, it turns out that there are situations where some strong conditions are fulfilled but
the Gruson-Jensen duality does not hold. We show that for a coalgebra which is only left or only right semiperfect, the
Gruson-Jensen duality can fail. We also show that another set of strong conditions on a coalgebra $C$ is not enough to
have such a duality: we give an example where $C^*$ is left and right almost noetherian (meaning that cofinite left
ideals and cofinite right ideals are finitely generated) and moreover $C^*$ is even noetherian on one side, but the
Gruson-Jensen duality between the functor rings of $\Mm^C$ and ${}^C\Mm$ does not hold.

\section{Finitely accessible categories and functor rings}

Throughout the paper all categories will be additive and all modules will be unitary. Let us recall some terminology on
finitely accessible categories. An object $P$ of a category $\mathcal{C}$ with direct limits is called \emph{finitely
presented} if the functor $\Hom_{\mathcal{C}}(P,-)$ commutes with direct limits. A category $\mathcal{C}$ is called
\emph{finitely accessible} (or \emph{locally finitely presented} in the terminology of \cite{CB}) if $\mathcal{C}$ has
direct limits, the class $\fp(\mathcal{C})$ of finitely presented objects of $\mathcal{C}$ is skeletally small, and
every object of $\mathcal{C}$ is a direct limit of finitely presented objects \cite{Prest}. A Grothendieck category is
finitely accessible if and only if it has a family of finitely presented generators. A finitely accessible category is
called \emph{locally coherent} if every finitely presented object is coherent, that is, finitely generated subobjects of
finitely presented objects are finitely presented. Such a category is necessarily Grothendieck and has a family
of finitely presented generators \cite{Prest}.

Now let $\mathcal{C}$ be a finitely accessible category and let $(U_i)_{i\in I}$ be a family of 
representatives of the isomorphism classes of finitely presented objects of $\mathcal{C}$. We associate a ring $R$ to
the family $(U_i)_{i\in I}$ in the following way (e.g., see \cite{DG}, \cite{Gabriel}): \[R=\bigoplus_{i\in
I}\bigoplus_{j\in J}\Hom_{\mathcal{C}}(U_i,U_j)\] as abelian group, and the multiplication is given by the rule: if
$f\in \Hom_{\mathcal{C}}(U_i,U_j)$ and $g\in \Hom_{\mathcal{C}}(U_k,U_l)$, then $fg=f\circ g$ if $i=l$ and zero
otherwise. Then $R$ is a ring with enough idempotents \cite{Fuller}, say $R=\bigoplus_{i\in I}e_iR=\bigoplus_{i\in
I}Re_i$. The idempotents $e_i$ are the elements of $R$ which are the identity on $U_i$ and zero elsewhere, and they form
a complete family of pairwise orthogonal idempotents. The ring $R$ constructed above is called the \emph{functor ring}
of $\mathcal{C}$. The family $(Re_i)_{i\in I}$ is a family of finitely generated projective generators of the category
${}_R\mathcal{M}$ of (unitary) left $R$-modules. A (unitary) left $R$-module $X$ is finitely presented if and only
if there is an exact sequence $\bigoplus_{i\in F_1}Re_i\to \bigoplus_{j\in F_2}Re_i\to X\to 0$ for some finite sets
$F_1$ and $F_2$ of indices in $I$. Now denote $U=\bigoplus_{i\in I}U_i$. Since $\Hom_{\mathcal{C}}(U_i,U)=Re_i$, it follows
that a left $R$-module $X$ is finitely presented if and only if there is an exact sequence $\Hom_{\mathcal{C}}(N,U)\to
\Hom_{\mathcal{C}}(M,U)\to X\to 0$. It is straightforward to show that the Yoneda functor
$\Hom_{\mathcal{C}}(-,U):\fp(\mathcal{C})\to {}_R\mathcal{M}$ is a contravariant full, faithful and left exact
functor, which reflects monomorphisms to epimorphisms and epimorphisms to split monomorphisms. It induces a duality
between finitely presented objects in $\mathcal{C}$ and finitely generated projective objects in ${}_R\Mm$.

We recall a couple of properties of finitely accessible categories which will be needed later on.

\begin{proposition}{\cite[Theorem~6.1]{Prest}}\label{fginfp} A finitely accessible category has products if and only if
the category of left modules over its functor ring is locally coherent. 
\end{proposition}

\begin{proposition}{\cite[Proposition~2.2]{Roos}}\label{AbCat} A finitely accessible category $\mathcal{C}$ is locally
coherent if and only if $\fp(\mathcal{C})$ is abelian (with only finite coproducts). 
\end{proposition}

The reader is referred to \cite{CB} and \cite{Prest} for more information on finitely accessible categories.

\section{Associated Freyd categories}

Let $\mathcal{C}$ be an additive category. Following \cite{Bel}, we shall associate to $\mathcal{C}$ two additive
categories which are defined as follows. In the morphism category ${\rm Mor}(\mathcal{C})$ of $\mathcal{C}$ denote an
object $u:M\to N$ by $(M,u,N)$ and a morphism by $(f,g):(M',u',N')\to (M,u,N)$, where $f:M'\to M$ and $g:N'\to N$ are
such that $uf=gu'$. Consider in ${\rm Mor}(\mathcal{C})$ the full subcategories $\mathcal{X}$ and $\mathcal{Y}$ 
consisting of all split monomorphisms and all split epimorphisms respectively. Then the stable categories
$\mathcal{B}(\mathcal{C})={\rm Mor}(\mathcal{C})/\mathcal{X}$ and $\mathcal{A}(\mathcal{C})={\rm
Mor}(\mathcal{C})/\mathcal{Y}$ are called the \emph{Freyd categories} associated to $\mathcal{C}$. An object $(M,u,N)$
from ${\rm Mor}(\mathcal{C})$ will be denoted by $[M,u,N]$ and $\{M,u,N\}$ when viewed as an object in
$\mathcal{B}(\mathcal{C})$ and $\mathcal{A}(\mathcal{C})$ respectively. Also, a morphism $(f,g)$ from ${\rm
Mor}(\mathcal{C})$ will be denoted by  $[f,g]$ and $\{f,g\}$ when viewed in $\mathcal{B}(\mathcal{C})$ and
$\mathcal{A}(\mathcal{C})$ respectively.  The category $\mathcal{B}(\mathcal{C})$ may be viewed alternatively as ${\rm
Mor}(\mathcal{C})$ modulo the congruence generated by the subgroup of $\Hom_{\mathcal{C}}(M,N)$ consisting of all
morphisms $(f,g):(M',u',N')\to (M,u,N)$ for which there exists a morphism $\alpha:N'\to M$ such that $\alpha u'=f$, i.e.
by commuting square morphisms factoring as in the following diagram with the lower left triangle and the big square
commuting (only):
$$\xymatrix{
M\ar[r]^u \ar@{..} @/_1pc/ [dr] | {=} & N \\
M'\ar[u]^f\ar[r]_{u'} & N'\ar[ul]_\alpha\ar[u]_g
}$$
Indeed, if $(f,g):(M',u',N')\to (M,u,N)$ is a morphism in ${\rm Mor}(\mathcal{C})$, then $[f,g]=0$ in
$\mathcal{B}(\mathcal{C})$ if and only if there is a morphism $\alpha:N'\to M$ such that $\alpha u'=f$. In particular
$[M,u,N]=0$ if and only if $f$ is a split monomorphism. Therefore, $\mathcal{B}(\mathcal{C})$ is one of the homotopy
categories introduced by Freyd in \cite{Freyd}. 

We note that the subgroup of $\Hom_{\mathcal{C}}(M,N)$ determining the above congruence is the sum of two groups, namely
the group of all morphisms $(f,g):(M',u',N')\to (M,u,N)$ for which there exists a morphism $\alpha:N'\to M$ such that
$\alpha u'=f$ and $u\alpha=g$ and the group of all morphisms $(f,g):(M',u',N')\to (M,u,N)$ with $f=0$. This follows
since if $f=\alpha u'$ then $(f,g)=(\alpha u',u\alpha)+(0,g-u\alpha)$. It is also worth to note that the equivalence
relation corresponding to the first of these groups is just the usual homotopy equivalence of chain complexes,
restricted to bounded chain complexes of length $2$. 

Similarly, $\mathcal{A}(\mathcal{C})$ is the other homotopy category introduced by Freyd in \cite{Freyd}, since it may
be viewed alternatively as ${\rm Mor}(\mathcal{C})$ modulo the congruence generated by the subgroup of 
$\Hom_{\mathcal{C}}(M,N)$ (for each pair of objects $(M,N)$) consisting of all morphisms $(f,g):(M',u',N')\to (M,u,N)$
for which there exists a morphism
$\beta:N'\to M$ such that $u\beta=g$, i.e. by commuting square morphisms factoring as in the following diagram with the
upper right triangle and the big square commuting (only): 
$$\xymatrix{
M\ar[r]^u \ar@{..} @/^1pc/ [dr] | {=} & N \\
M'\ar[u]^f\ar[r]_{u'} & N'\ar[ul]^\beta\ar[u]_g
}$$

The Freyd categories are related as follows. We shall denote equivalences of categories by ``$\approx$''. 

\begin{proposition}{\cite[Proposition~3.6]{Bel}} \label{acbc} Let $\mathcal{C}$ and $\mathcal{D}$ be two
categories. Then:

(i) $\mathcal{C}\approx \mathcal{D}$ if and only if $\mathcal{A}(\mathcal{C})\approx
\mathcal{A}(\mathcal{D})$ if and only if $\mathcal{B}(\mathcal{C})\approx \mathcal{B}(\mathcal{D})$.

(ii) $\mathcal{A}(\mathcal{C}^{\rm op})\approx \mathcal{B}(\mathcal{C})^{\rm op}$ and
$\mathcal{B}(\mathcal{C}^{\rm op})\approx \mathcal{A}(\mathcal{C})^{\rm op}$. 
\end{proposition}

\begin{proposition}\label{Dualities} Let $\mathcal{C}$ be an abelian category. 

(i) The category $\mathcal{B}(\mathcal{C})$ is an abelian category, equivalent to the category of exact chain complexes
of type $M\rightarrow N\rightarrow P\rightarrow 0$ of objects in $\mathcal{C}$ up to the usual homotopy equivalence.

(ii) The category $\mathcal{A}(\mathcal{C})$ is an abelian category, equivalent to the category of exact chain complexes
of type $0\rightarrow K\rightarrow M\rightarrow N$ of objects in $\mathcal{C}$ up to the usual homotopy equivalence.
\end{proposition}

\begin{proof} (i) The category $\mathcal{B}(\mathcal{C})$ is abelian by \cite[Proposition~4.5]{Bel}. The second part of
(i) follows by considering the functor given on objects by \[[M,u,N]\,\, \longmapsto \,\,
(M\stackrel{u}{\rightarrow}N\rightarrow \Coker(u)\rightarrow 0)\] and observing that any morphism
$[f,g]:[M',u',N']\to [M,u,N]$ in $\mathcal{B}(\mathcal{C})$ can be extended to a morphism
$h:\Coker(u')\rightarrow \Coker(u)$. If $[f,g]=0$, then there is $\alpha:N'\to M$ such that $\alpha u'=f$. Since
$(u\alpha -g)u'=u\alpha u'-gu'=uf-gu'=0$, we see that $\Ker(p')=\im(u')\subseteq \Ker(u\alpha-g)$ so then there is 
$\alpha':\Coker(u)\to N$ such that $\alpha'p'+u\alpha=g$: 
$$\xymatrix{
\ar@{} & M\ar[r]^u \ar@{..}
@/_1pc/ [dr] | {=} \ar@{.}@/_1.4pc/[dr] \ar@{.}@/_/[dr]            & N \ar[r]^(.4)p         \ar@{..} @/^1pc/ [dr] | {=}
\ar@{.}@/^1.5pc/[dr] \ar@{.}@/^/[dr]   & \Coker(u) \ar[r] & 0  & \,\\
 & M'\ar[u]^f\ar[r]_{u'} & N'\ar[ul]|\alpha\ar[u]_g|= \ar[r]_(.4){p'}   \ar@/_1pc/@{.}[u] | + \ar@/^0.7pc/@{.}[u] | +  
& \Coker(u')\ar@{..>}[ul] | {\alpha'}\ar[u]_{h}\ar[r] & 0 & \,
}$$
Moreover, we have that $hp'=pg=p\alpha'p'+pu\alpha=p\alpha'p'$, and so $h=p\alpha'$ ($p'$ is an epimorphism), so
$(f,g,h)$ is a null-homotopic morphism of chain complexes. The inverse functor is obvious. 

(ii) is analogous to (i).

We note that the result of \cite{Bel} that $\mathcal{A}(\mathcal{C})$ and $\mathcal{B}(\mathcal{C})$ are
abelian also follows by the above equivalences.
\end{proof}

The full subcategories of finitely presented left or right modules over the functor ring of a locally coherent
category $\mathcal{C}$ are closely related to the Freyd categories associated to $\mathcal{C}$ (for a general case, see
\cite[Corollary~3.9]{Bel}).

\begin{corollary} \label{fp-htpy} Let $\mathcal{C}$ be a locally coherent category and let $R$ be its functor ring. Then
 ${\rm fp}(\Mm_R)\approx \mathcal{A}({\rm fp}(\mathcal{C}))$ and ${\rm fp}({}_R\Mm)\approx \mathcal{B}({\rm
fp}(\mathcal{C}))^{\rm op}$.
\end{corollary}

\begin{proof} We sketch the second part. Note that $\fp(\mathcal{C})$ is abelian by Proposition \ref{AbCat}. A finitely
presented left $R$-module is the cokernel of a morphism between finitely generated left $R$-modules. Using the duality
between the categories of finitely generated projective left $R$-modules and finitely generated projective right
$R$-modules, the equivalence between the category finitely generated projective right $R$-modules and
$\fp(\mathcal{C})$, and Proposition \ref{Dualities}, we obtain a duality between ${\rm fp}({}_R\Mm)$ and
$\mathcal{B}({\rm fp}(\mathcal{C}))$. The duality is explicitly given on objects as follows. Let $(U_i)_{i\in I}$ be a
family of representatives of the isomorphism classes of finite dimensional left $C$-comodules and denote
$U=\bigoplus_{i\in I}U_i$. We have seen that an object $X$ of ${}_R\Mm$ is finitely presented if and only if there is a
finite presentation $\Hom_{\mathcal{C}}(N,U)\to \Hom_{\mathcal{C}}(M,U)\to X\to 0$ for some objects $M$ and $N$ in
$\fp(\mathcal{C})$. This is induced by an object $[M,u,N]$ of $\mathcal{B}({\rm fp}(\mathcal{C}))$. 
\end{proof}

Now we can give an interpretation of symmetry of finitely accessible categories in terms of Freyd categories. 

\begin{definition} \rm Let $\mathcal{C}$ and $\mathcal{D}$ be two finitely accessible categories. We call $\mathcal{C}$
and $\mathcal{D}$ \emph{symmetric} categories if there is a duality between the categories
$\mathcal{B}({\rm fp}(\mathcal{C}))$ and $\mathcal{B}({\rm fp}(\mathcal{D}))$. 
\end{definition}

Let us note that $\mathcal{C}$ and $\mathcal{D}$ are symmetric in the above sense if and only if they are symmetric 
in the sense of Dung and Garc\'\i a \cite{DG}. To this end, let us denote by $R$ and $L$ the functor rings of
$\mathcal{C}$ and $\mathcal{D}$ respectively. Then by Proposition \ref{fp-htpy} we have $\mathcal{B}({\rm
fp}(\mathcal{C}))\approx \mathcal{B}({\rm fp}(\mathcal{D}))^{\rm op}$ if and only if $({\rm
fp}({}_R\mathcal{M}))^{\rm op}\approx {\rm fp}({}_L\mathcal{M})$ if and only if $\mathcal{C}$ and $\mathcal{D}$
are symmetric in the sense of \cite[Definition~2.8]{DG}.

\section{Coalgebras and comodules}

Now we recall several facts on coalgebras and comodules, mainly following \cite{DNR}. Let $C$ be a coalgebra over a
field $k$. Denote by $C^*=\Hom_k(C,k)$ the dual algebra of $C$ over $k$. Then $C^*$ is a topological vector space
endowed with the weak-$*$ topology, in which the closed subspaces are annihilators in $C^*$ of subspaces of $C$. 
In this topology, $C^*$ has a basis of neighbourhoods of $0$ consisting of ideals of finite codimension.
The coalgebra $C$ is called \emph{left $\mathcal{F}$-noetherian} if every closed and cofinite left ideal of $C^*$ (in
the sense that $C^*/I$ is finite dimensional) is finitely generated in $\mathcal{M}_{C^*}$ (see \cite{Radford} and
\cite{CNV}). In particular, every right semiperfect coalgebra is left $\mathcal{F}$-noetherian
\cite[Theorem~2.12]{CNV}. The coalgebra $C$ is called \emph{left strongly reflexive} or \emph{left almost noetherian} 
(or $C^*$ is \emph{left almost noetherian}, see \cite{Radford} and
\cite{CNV}) if every cofinite left ideal $I$
of $C^*$ is finitely generated \cite{CNV}.  Clearly, every left almost noetherian coalgebra is
left $\mathcal{F}$-noetherian.

A right $C^*$-module $M$ is called \emph{rational} if for every $x\in M$, there are $x_1,\dots,x_n\in M$ and
$c_1,\dots,c_n\in C$ such that $$xc^*=\sum_{i=1}^nx_ic^*(c_i)$$ for every $c^*\in C^*$. The class
$\Rat(\mathcal{M}_{C^*})$ of rational right $C^*$-modules is closed under submodules, direct sums, direct products and
homomorphic images. In fact, $\Rat(\mathcal{M}_{C^*})=\sigma[C_{C^*}]$, where $\sigma[C_{C^*}]$ is the full subcategory
of the category $\mathcal{M}_{C^*}$ of right $C^*$-modules consisting of the modules subgenerated by $C$. 

Denote by ${}^C\mathcal{M}$ the category of left $C$-comodules. Then there is an isomorphism of categories
${}^C\mathcal{M}\cong \Rat(\mathcal{M}_{C^*})$. We shall frequently make the identification between left $C$-comodules
and rational right $C^*$-modules. The category of rational right $C^*$-modules, and so the category of left
$C$-comodules, is a Grothendieck category, which has a family of finite dimensional generators, namely the rational
right $C^*$-modules of the form $C^*/I$ with $I$ a closed cofinite (two-sided) ideal of $C^*$. It is easy to note that a
$C$-comodule is finitely presented if and only if it is finitely generated if and only if it is finite dimensional. We
denote by $\fd({}^C\Mm)$ the class of finite dimensional left $C$-comodules. Similar considerations may be made for the
categories $\mathcal{M}{}^C$ of right $C$-comodules and $\Rat({}_{C^*}\mathcal{M})$ of rational left $C^*$-modules. 
Note that the functor $(-)^*$ defines a duality between the categories $\fd({}^C\Mm)$ and  $\fd(\Mm^C)$ \cite{DNR}. 
We also refer to \cite{DNR} and \cite{BW} for basics on coalgebras and their comodules.

We note an interesting characterization of the category $\Mor(\fd(\Mm^C))$. One can prove without much
difficulty that this category is equivalent (in fact, isomorphic) to the category of finite dimensional right comodules
over the upper triangular matrix coalgebra 
$M_\Delta^2(C)=\left(
\begin{array}{cc}
	C & C \\
	0 & C
\end{array}
\right)$
with comultiplication and counit given by (we use the Sweedler notation with the summation symbol omitted)
\begin{eqnarray*}
\left(\begin{array}{cc} x & y \\ 0 & z \end{array} \right) & \longmapsto & 
\left(\begin{array}{cc} x_1 & 0 \\ 0 & 0 \end{array} \right)\otimes \left(\begin{array}{cc} x_2 & 0 \\ 0 & 0 \end{array}
\right) + \left(\begin{array}{cc} y_1 & 0 \\ 0 & 0 \end{array} \right) \otimes \left(\begin{array}{cc} 0 & y_2 \\ 0 & 0
\end{array} \right) + \\
& & + \left(\begin{array}{cc} 0 & y_1 \\ 0 & 0 \end{array} \right) \otimes \left(\begin{array}{cc} 0 & 0 \\ 0 & y_2
\end{array} \right) + \left(\begin{array}{cc} 0 & 0 \\ 0 & z_1 \end{array} \right) \otimes \left(\begin{array}{cc} 0 & 0
\\ 0 & z_2 \end{array} \right), \\
\left(\begin{array}{cc} x & y \\ 0 & z \end{array} \right) & \longmapsto & \varepsilon_C(x)\varepsilon_C(z).
\end{eqnarray*}
The dual algebra of this coalgebra is precisely the algebra $M_{2}^\Delta(C^*)$ of upper triangular matrices with
entries in $C^*$. The stated equivalence associates to any $M\stackrel{u}{\longrightarrow}N$ the $M^2_\Delta(C)$-comodule $M\oplus
N$ with coaction given by:
\begin{eqnarray*}
\left(\begin{array}{c} m \\ n  \end{array}\right) & \longmapsto & \left(\begin{array}{c} n_0 \\ 0 \end{array} \right)
\otimes \left(\begin{array}{cc} n_1 & 0 \\ 0 & 0 \end{array} \right) + \left(\begin{array}{cc} f(m_0) \\ 0 \end{array}
\right) \otimes \left(\begin{array}{cc} 0 & m_1 \\ 0 & 0 \end{array} \right) + \left(\begin{array}{cc} 0  \\ n_0 
\end{array} \right) \otimes \left(\begin{array}{cc} 0 & 0 \\ 0 & m_1 \end{array} \right).
\end{eqnarray*}
The inverse of this equivalence goes as follows. Denote \[E=\left(\begin{array}{cc} 0 & 0 \\ 0 & \varepsilon
\end{array} \right), \quad N=\left(\begin{array}{cc} 0 & \varepsilon \\ 0 & 0 \end{array} \right), \quad 
F=\left(\begin{array}{cc} \varepsilon & 0 \\ 0 & 0 \end{array} \right)\] as elements of $M_2^\Delta(C^*)$. Then, for a
right $M^2_\Delta(C)$-comodule $T$, we have $T=F\cdot T\oplus E\cdot T$ and the inverse functor associates the object
$N:E\cdot T\rightarrow F\cdot T$ in $\Mor(\fd(\Mm^C))$, with the morphism $N$ given by left multiplication by $N\in
M_2^\Delta(C^*)=(M^2_\Delta(C))^*$. The correspondence on morphisms is similar.

\section{Positive answers}

In what follows we shall discuss symmetry for comodule categories, or equivalently, the existence of a Gruson-Jensen
duality in the case of comodule categories. 

As noted in the introduction, the case of a (left and right) semiperfect coalgebra can be solved by reducing it to a
result for module categories over rings with enough idempotents. We now give some details. 

\begin{theorem} \label{t:semiperf} Let $C$ be a (left and right) semiperfect coalgebra. Then the categories ${}^C\Mm$
and $\Mm^C$ are symmetric. More precisely, there is a duality between $\fp({}_R\Mm)$ and $\fp(\Mm_R)$. \\
Moreover, if $C$ is a right semiperfect coalgebra and the categories ${}^C\Mm$ and $\Mm^C$ are symmetric, 
then $C$ is (left and right) semiperfect.
\end{theorem}

\begin{proof} Denote by $R$ and $L$ the functor rings of ${}^C\Mm$ and $\Mm^C$ respectively. Since $C$ is (left and
right) semiperfect, it is known and not difficult to show (see, for example, \cite[Chapter 3]{DNR}) that
$\Rat({}_{C^*}C^*)=\Rat(C^*_{C^*})$ is an idempotent ideal of $C^*$ denoted simply $\Rat(C^*)$. Moreover, $\Rat(C^*)$ is
a ring with enough idempotents, and the category of rational left $C^*$-modules (right $C$-comodules) is the same as
that of left $\Rat(C^*)$-modules (see also \cite{BW}). Similarly, the category of left $C$-comodules is the same as that
of right $\Rat(C^*)$-modules. Therefore, applying \cite[Theorem~2.9]{DG} (which solves the case of unitary modules over
a ring with enough idempotents), there is a duality between $\fp({}_R\Mm)$ and $\fp(\Mm_{L^{\rm op}})$. But $R\simeq
L^{\rm op}$ since \[R=\bigoplus\limits_{M,N\in {}^C\Mm}{}^C\Hom(M,N)\simeq
\bigoplus\limits_{M,N\in{}^C\Mm}{}^C\Hom(N^*,M^*)=\bigoplus\limits_{P,Q\in\Mm^C}\Hom^C(P,Q)=L^{\rm op}\] and, 
because $(-)^*$ is a contravariant functor, the multiplication of $R$ (composition) is reverted.\\
For the last part, note that for a right semiperfect coalgebra the category $\Mm^C$ has a generating family of small
projective objects (see \cite[Chapter 3]{DNR}), and so, by well known results of category theory 
(e.g. see \cite{H}) it is equivalent to the category of unital left $A$-modules for a ring $A$ with enough idempotents.
In this case, its symmetric category ${}^C\Mm$ must be equivalent to the category of right unital $A$-modules by
\cite[Theorem 2.9]{DG}, and so it has a generating set of projective objects. Therefore, $C$ is also left semiperfect.
\end{proof}

\begin{proposition} The categories ${}^C\Mm$ and $\Mm^C$ are symmetric if and only if $\mathcal{A}({\rm
fd}({}^C\Mm))\approx \mathcal{B}({\rm fd}({}^C\Mm))$ if and only if $\mathcal{A}({\rm
fd}(\Mm^C))\approx \mathcal{B}({\rm fd}(\Mm^C))$. 
\end{proposition}

\begin{proof} Using the duality between $\fd({}^C\Mm)$ and $\fd(\Mm^C)$ and Proposition \ref{acbc}, we have
$\mathcal{B}({\rm fd}(\Mm^C))\approx \mathcal{B}({\rm fd}({}^C\Mm)^{\rm op})\approx \mathcal{A}({\rm
fd}({}^C\Mm))^{\rm op}$. Now it follows that ${}^C\Mm$ and $\Mm^C$ are symmetric if and only if there is a duality
between $\mathcal{B}({\rm fd}(\Mm^C))$ and $\mathcal{B}({\rm fd}({}^C\Mm))$ if and only if $\mathcal{A}({\rm
fd}({}^C\Mm))\approx \mathcal{B}({\rm fd}({}^C\Mm))$. The last assertion follows in a similar way. 
\end{proof}

\section{Negative answers}

In this section we will study the simple objects of the category $\fp({}_R\Mm)$ of finitely presented left
$R$-modules over the functor ring $R$ of a locally finite category and use the conclusions to give
examples of situations where the Gruson-Jensen duality does not happen. 

\begin{proposition}\label{simple1} Let $\mathcal{C}$ be a finitely accessible category with products having functor ring
$R$. Then a finitely presented left $R$-module $X$ is simple in $\fp({}_R\Mm)$ if and only if it is simple in
${}_R\Mm$. 
\end{proposition}

\begin{proof} By Proposition \ref{fginfp}, $\fp({}_R\Mm)$ is locally coherent. Then by Proposition \ref{AbCat},
$\fp({}_R\Mm)$ is abelian (with only finite coproducts). It is easy to deduce that a morphism between finitely presented
left $R$-modules is a monomorphism (epimorphism) in $\fp({}_R\Mm)$ if and only if it is a monomorphism (epimorphism) in
${}_R\Mm$, and kernels and cokernels are computed in ${}_R\Mm$. Then the subobjects of an object in $\fp({}_R\Mm)$
are the finitely presented submodules. Therefore, a simple finitely presented left $R$-module will be a simple
object in $\fp({}_R\Mm)$.

Conversely, assume that $X$ is simple in $\fp({}_R\Mm)$. Denote $U=\bigoplus_{i\in I}U_i$, where $(U_i)_{i\in I}$ is
a family of representatives of the isomorphism classes of finitely presented objects of $\mathcal{C}$. If $X$ is not a
simple left $R$-module, then there is some $x\in X$ such that $0\neq Rx\neq X$. But since $X$ is unitary, $e_ix=x$ for
some idempotent $e_i$ corresponding to a split inclusion $U_i\hookrightarrow U$, and so in fact $Rx$
is finitely generated and is a quotient of $Re_i$. Since $\fp({}_R\Mm)$ is locally coherent, $Rx$ has to be finitely
presented, and this contradicts the assumption that $X$ contains no non-trivial finitely presented submodules.
\end{proof}

A finitely accessible category will be called \emph{locally finite} if every finitely presented object has finite
length. In particular, it is locally coherent, and so Grothendieck. The following result is important as it computes 
the simple objects of the abelian category $\mathcal{B}(\fp(\mathcal{C}))$ for a locally finite category $\mathcal{C}$,
and it will be the key ingredient that we will use in order to construct examples of non-symmetric comodule categories. 

\begin{proposition}\label{simple2} Let $\mathcal{C}$ be a locally finite category with functor ring $R$. An object $X$
in $\fp({}_R\Mm)$ is simple if and only if there is an object $u:M\rightarrow N$ in $\mathcal{B}(\fp(\mathcal{C}))$,
which is the corresponding object of $X$ (up to equivalence) through the duality of Corollary \ref{fp-htpy},
and such that $M\rightarrow N$ is an epimorphism with simple kernel and $M$ is indecomposable injective. 
\end{proposition}

\begin{proof} Denote $U=\bigoplus_{i\in I}U_i$, where $(U_i)_{i\in I}$ is a family of representatives of the isomorphism
classes of finitely presented objects of $\mathcal{C}$. Consider a finite presentation
\[\diagram \Hom(N,U)\rto^{\Hom(u,U)} & \Hom(M,U)\rto^{\qquad \varphi}\rto & X\rto & 0\enddiagram\] of the simple object
$X$ for some objects $M, N$ of $\fp(\mathcal{C})$. Then $u:M\to N$ is an object in $\mathcal{B}(\fp(\mathcal{C}))$. 

Note that we may assume that $M$ is indecomposable. Indeed, write $M=\bigoplus\limits_{i\in F}M_i$ as a finite direct
sum of indecomposables. Then since $\varphi\neq 0$ at least one of its restrictions to $\Hom(M_i,U)$ is nonzero, and
since $X$ is simple, this restriction will be surjective. Moreover, the kernel of this restricted morphism (as a
morphism of left $R$-modules) will be finitely presented, so in particular finitely generated.

Let $\sigma_M$ be the (split) inclusion of $M$ into $U$. The fact that $X$ is simple translates equivalently to the
following: for any $\alpha\in \Hom(M,U)$, either $\alpha$ factors through $u:M\to N$ (so $\alpha=\overline{\alpha}u$ for
some $\overline{\alpha}\in R$) or there are $\beta\in R$ and $\gamma\in\Hom(N,U)$ such that
$\sigma_M=\beta\alpha+\gamma u$ (i.e. the generator $\sigma_M$ of $\Hom(M,U)$ is generated by $\alpha$ modulo
$\Hom(N,U)$, for any $\alpha\neq 0$ modulo
$\Hom(N,U)$). 
$$\xymatrix{
M\ar[d]_\alpha\ar@{^{(}-}[dr]_{\sigma_M} | {\bigoplus}\ar[r]^u & N\ar[d]^\gamma & & 
M\ar[d]_\alpha\ar@{=}[dr] | *+[o][F-]{1_M}\ar[r]^u\ar@{.}@/^1.2pc/[dr] | {+} \ar@{.}@/_1.2pc/[dr] | {+} & N\ar[d]^\gamma
\\
U\ar[r]_(0.3)\beta & U=M\oplus U'_M & & P\ar[r]_(0.5)\beta & M 
}$$
Now, since $\alpha$ has finite image in $U$, and $M$ is a direct summand with complement $U'_M$, and 
$\beta$ and $\gamma$ also have finite images, this condition is equivalent to the one given by the right diagram above,
where $P$ can be any object of finite length of $\mathcal{C}$. That is, whenever $\alpha$
does not factor through $u$ there are $\beta$ and $\gamma$ such that $1_M=\beta\alpha+\gamma u$. We have two cases:

{\bf $\bullet$ (1)} $\im(u)=N$. Let $K=\Ker(u)$. We note that a quotient $\alpha:M\rightarrow M/L$ splits through $u$ 
if and only if $K\subseteq L$. When this is not true, equivalently, when $L\cap K\neq K$, we see that $L\cap K\subseteq
\Ker(\beta\alpha)\cap\Ker(\gamma u)$, and since $\beta\alpha+\gamma u=1_M$, we get $L\cap K=0$. This shows that $K$ is a
simple subobject of $M$ (pick $L$ any subobject of $K$ for this purpose); if $K=0$, then we would have $X=0$, a
contradiction. Now choose an arbitrary inclusion morphism $\alpha:M\rightarrow P$ and find such $\beta,\gamma$ as above
(because $\alpha$ does not factor through $u$). The equality $1_M=\beta\alpha+\gamma u$ shows that the morphism
$h=(\alpha+u):M\rightarrow P\oplus N$ is split by $(\beta+\gamma)$. So we have $h(M)\oplus M'=N\oplus
P=(\bigoplus\limits_{k\in K}N_k)\oplus(\bigoplus\limits_{j\in J}P_j)$, with each $N_k$ and $P_j$ indecomposable. Now,
since all these objects are of finite length, by (an equivalent form of the) Krull-Remak-Schmidt-Azumaya theorem, we
find that $h(M)$ is a complement of a direct sum of some subfamily of $\{N_k,P_j\}_{k\in K, j\in J}$ (see e.g.
\cite[Lemma~9.2.2]{N} or \cite{AF}). Since $h(M)$ is indecomposable, and it
cannot be isomorphic to any of the $N_k$'s because ${\rm length}(M)\geq{\rm length}(N_k)$ ($u$ is an epimorphism with
nonzero kernel) we get $h(M)\oplus N\oplus (\bigoplus\limits_{j\neq j_0}P_j)=N\oplus P$, for $j_0\in J$ for which
$M\simeq h(M)\simeq P_{j_0}$. In particular, $h(M)\cap (N\oplus (\bigoplus\limits_{j\neq j_0}P_j))=0$, and therefore, if
$p$ is the projection onto $P_{j_0}$ we obtain $h(M)\cap \Ker(p)=0$, so $ph$ is a monomorphism. Thus, as $M\simeq
P_{j_0}$, $ph$ is an isomorphism. But note that $ph=p(\alpha+u)=p\alpha+pu=p\alpha$, since $\im(u)\subseteq N\subseteq
\Ker(p)$. This shows that $M$ splits off in $P$, and therefore it is injective in the category $\fp(\mathcal{C})$.
Using, similar arguments as those in \cite[Chapter~2, Section~4]{DNR}, it follows that $M$ is injective in
$\mathcal{C}$.

{\bf $\bullet$ (2)} $\im(u)\neq N$. We first note that $u$ must be a monomorphism. Consider
$\alpha:M\longrightarrow P=\im(\alpha)$ the corestriction of $u$. If this splits through $u$, then it is easy to see
that $\im(u)$ is a direct summand in $N$. This situation reduces to the previous one, since if $u:M\rightarrow
\im(u)\oplus T$, then we can have an exact sequence $\Hom(\im(u),U)\rightarrow \Hom(M,U)\rightarrow X\rightarrow 0$
(i.e. then $M\rightarrow N$ and $M\rightarrow \im(u)$ represent the same object of $\Mor(\fp(\mathcal{C}))$). As before,
finding $\beta$ and $\gamma$ with $\beta\alpha+\gamma u=1_M$ will yield that $\Ker(u)=0$, since $\Ker(u)=\Ker(\alpha)$.
But now we have a context similar to that of the proof of (1), with the roles of $\alpha$ and $u$ reversed: $u$ is a
monomorphism and $\alpha$ is an epimorphism. As before, we get that $u$ splits; but this situation is not
possible, since in this case, $u$ defines the zero object of $\mathcal{B}(\fp(\mathcal{C}))
\,\simeq (\fp({}_R\Mm))^{\rm op}$ (see Proposition \ref{fp-htpy}), i.e. $X=0$. 

Finally, let $u:M\rightarrow N$ an epimorphism with simple kernel $S$ and $M$ finite dimensional indecomposable
injective. We check that the finitely presented object $X$ corresponding to $u:M\rightarrow N$ is simple, and for
this we check the equivalent condition given by the above right diagram. Let $\alpha:M\rightarrow P$ be a morphism. If
$\Ker(\alpha)\supseteq S$, then obviously $\alpha$ factors through $u$. Otherwise, $\Ker(\alpha)\cap S=0$, and since $M$
is indecomposable injective and $S$ is simple, $S$ is the socle of $M$ and is essential in $M$. Thus we have
$\Ker(\alpha)=0$ (since $\Ker(\alpha)\cap S=0$), so $\alpha$ is injective. But then, since $M$ is an injective object,
$\alpha$ splits off and we can find $\beta$ with $\beta\alpha=u$, and so we can take $\gamma=0$.
\end{proof}

Next we recall the following definition and results from \cite{CGT,I,LS}.

\begin{definition} \rm A right $C$-comodule $M$ is called \emph{chain} (or \emph{uniserial}) if the lattice of its right
subcomodules is a chain. A coalgebra $C$ is called \emph{right serial} if its indecomposable injective right comodules
are uniserial, equivalently, $C$ is a direct sum of uniserial right comodules. $C$ is called \emph{serial} if it is both
left and right serial.
\end{definition}

It is shown in \cite{CGT} that a coalgebra is serial (see also \cite[Corollary 25.3.4]{F} and the proof of
\cite[Proposition 4.4]{I}) if and only if any finite dimensional right (equivalently, any finite dimensional left)
$C$-comodule is a direct sum of chain comodules. 

Recall that when $X$ is a locally finite partially ordered set and $F$ is a field (so for each $x\leq y$ there are only
finitely many $z$ such that $x\leq z\leq y$), $C_X=F\{(x,y)\in X\times X | x\leq y \}$, the vector space 
with basis $\{(x,y)\in X\times X | x\leq y \}$ becomes an $F$-coalgebra when
endowed with the comultiplication $\Delta((x,y))=\sum\limits_{x\leq z\leq y}(x,z)\otimes (z,y)$ and counit
$\varepsilon((x,y))=\delta_{(x,y)}$. Let $C_\NN$ be the coalgebra associated in this way to the set of natural numbers.
Then we have:

\begin{proposition} $C_\NN=\bigoplus\limits_{n\in \NN}F\{(n,p)|p\geq n\}$ is a decomposition of $C_\NN$ into
indecomposable injective right comodules, and $C_\NN=\bigoplus\limits_{n\in\NN}F\{(k,n)|k\leq n\}$ is a decomposition of
$C_\NN$ as a direct sum of indecomposable injective left comodules. Moreover, all these are chain comodules.
Consequently, $C_\NN$ is a serial coalgebra. It is right semiperfect, and not left semiperfect.
\end{proposition}

\begin{proof} It is easy to see that $E_r(n)=F\{(n,p)|p\geq n\}$ are right subcomodules and $E_l(n)=F\{(k,n)|k\leq n\}$
are left subcomodules. Also note that the coradical of this coalgebra is cocommutative, and consists of the span of all
grouplike elements $F\{(n,n)|n\in\NN\}$. Therefore, the socle (the simple part) of each $E_r(n)$ (and $E_l(n)$) is
precisely $F\{(n,n)\}$, so it is simple. This shows that they are indecomposable injective (since they are summands of
$C_\NN$). We also see that $E_r(n)/F\{(n,n)\}\simeq E_r(n+1)$ and $E_l(n)/F\{(n,n)\}\simeq E_l(n-1)$. Then, using an
inductive process, we show that the Loewy filtration of each $E_r(n)$ has simple quotient at each step, so
$E_r(n)$ are chain comodules by \cite{I}. Similarly, $E_l(n)$ are chain comodules. \\
This coalgebra is right semiperfect since the injective
indecomposable left comodules $E_l(n)$ are finite dimensional, and not left semiperfect since there are 
infinite dimensional injective indecomposable right comodules $E_r(n)$ (in fact, all these are infinite dimensional).
\end{proof}

\begin{example} \rm The coalgebra $C_\NN$ does not have Gruson-Jensen duality. This follows immediately by applying
Theorem \ref{t:semiperf}, since this coalgebra is right but not left semiperfect. It also follows by applying the
results of this section, which provide a little more information about the categories of modules over the functor rings.
Indeed, note that there is a simple
injective left $C_\NN$-comodule, namely $F\{(0,0)\}$. By Proposition \ref{simple2}, there will be some simple modules in
the category of finitely presented left modules over the functor ring $R$ of ${}^C\Mm$. However, again by Proposition
\ref{simple2}, there will be no simple modules in the category of finitely presented left modules over the functor ring
$L\simeq R^{\rm op}$ of $\Mm^C$, because there is no finite dimensional injective right $C_\NN$-comodule, so there can
be no epimorphism $Q\rightarrow Q/T$ with $Q$ injective and $T$ a simple right subcomodule of $Q$ (which would be the
object from $\mathcal{B}(\fp(\mathcal{C}))$ corresponding to a simple finitely presented left $L$-module). This
illustrates 
the result of Theorem \ref{t:semiperf} by example.
\end{example}

In what follows, we show that other possible replacements of the semiperfect hypothesis of Theorem \ref{t:semiperf}
to very strong finitary properties are still not enough for this duality to hold. 

We recall the following general type of construction, dual to the generalized upper triangular matrix ring (see also
\cite[Section 4]{IL}). Let $C$ and $D$ be two coalgebras and $M$ be a $C$-$D$ bicomodule. Write $c\rightarrow c_1\otimes
c_2$, $d\rightarrow
d_1\otimes d_2$ for the comultiplications of $C$ and $D$ respectively, and $m\mapsto m_{-1}\otimes m_0$ and
$m\rightarrow m_0\otimes m_1$ the left and right coactions of $M$ (because of the bicomodule condition, there is no
danger of confusion even if both coactions are present). Then on the vector space $H=C\oplus M\oplus
D=\left(\begin{array}{cc} C & M \\ 0 & D \end{array} \right)$ we can introduce the coalgebra structure given by:
\begin{eqnarray*}
\left(\begin{array}{cc} c & m \\ 0 & d \end{array} \right) & \longmapsto & 
\left(\begin{array}{cc} c_1 & 0 \\ 0 & 0 \end{array} \right)\otimes \left(\begin{array}{cc} c_2 & 0 \\ 0 & 0 \end{array}
\right) 
+ \left(\begin{array}{cc} m_{-1} & 0 \\ 0 & 0 \end{array} \right) \otimes \left(\begin{array}{cc} 0 & m_0 \\ 0 & 0
\end{array} \right) + \\
& & + \left(\begin{array}{cc} 0 & m_0 \\ 0 & 0 \end{array} \right) \otimes \left(\begin{array}{cc} 0 & 0 \\ 0 & m_1
\end{array} \right) 
+ \left(\begin{array}{cc} 0 & 0 \\ 0 & d_1 \end{array} \right) \otimes \left(\begin{array}{cc} 0 & 0 \\ 0 & d_2
\end{array} \right) \\
\left(\begin{array}{cc} c & m \\ 0 & d \end{array} \right) & \longmapsto & \varepsilon_C(c)\varepsilon_D(d)
\end{eqnarray*}
Using these, it is easy to note that we have:
\begin{eqnarray*}
\left(\begin{array}{cc} C & M \\ 0 & D \end{array} \right) & = & \left(\begin{array}{cc} C & 0 \\ 0 & 0 \end{array} 
\right) \oplus \left(\begin{array}{cc} 0 & M \\ 0 & D \end{array} \right) \,\,\,\,\, {\rm as\,left\,comodules,} \\
\left(\begin{array}{cc} C & M \\ 0 & D \end{array} \right) & = & \left(\begin{array}{cc} C & M \\ 0 & 0 \end{array} 
\right) \oplus \left(\begin{array}{cc} 0 & 0 \\ 0 & D \end{array} \right) \,\,\,\,\, {\rm as\,right\,comodules.}
\end{eqnarray*} 
Also, one can see that $H^*=\left(\begin{array}{cc} C^* & M^* \\ 0 & D^* \end{array} \right)$, 
the usual upper triangular matrix ring with $M^*$ a $C^*$-$D^*$ bimodule. 

\begin{example}\label{ex.6} \rm 
Let $C$ be the divided power coalgebra over a field $F$, which is the finite dual of the algebra of formal power series
$F[[X]]$. It has a basis $(c_n)_{n\geq 0}$, and comultiplication $\Delta(c_n)=\sum\limits_{i+j=n}c_i\otimes c_j$ and
counit $\varepsilon(c_n)=\delta_{0n}$. Let $D=F$ as a coalgebra. Then $\varepsilon:C\rightarrow D$ is a morphism of
coalgebras, and so $M=C$ becomes a $C$-$D$ bicomodule (in fact the right $D$-comodule structure is nothing else
but the vector space structure of $C$).  Let $H=\left(\begin{array}{cc} C & M \\ 0 & D\end{array}
\right)=\left(\begin{array}{cc} C & C \\ 0 & F\end{array} \right)$ be the coalgebra defined above. More specifically,
this coalgebra has a basis $\{(c_n)_n; (x_n)_n; t\}$ with comultiplication $\Delta(c_n)=\sum\limits_{i+j=n}c_i\otimes
c_j$, $\Delta(x_n)=\sum\limits_{i+j=n}c_i\otimes x_j + x_n\otimes t$, $\Delta(t)=t\otimes t$ and counit given by
$\varepsilon(c_n)=\delta_{0n}$, $\varepsilon(t)=1$, $\varepsilon(x_n)=0$. Then:

{\bf $\bullet$} The decomposition of $H$ as indecomposable injective left comodules is $H=F\{(c_n)_n\}\oplus
F\{(x_n)_n;t\}$.

{\bf $\bullet$} The decomposition of $H$ as indecomposable injective right comodules is $H=F\{(c_n)_n;
(x_n)_n\}\oplus Ft$.

Denote by $R$ and $L$ the functor rings of ${}^H\Mm$ and $\Mm^H$ respectively. As before, by Proposition \ref{simple2},
this shows that there are no simple modules in $\fp({}_R\Mm)$, because there are no finite dimensional injective
comodules in ${}^H\Mm$, but there are some simple modules in $\fp({}_L\Mm)$
(corresponding to the simple injective right $H$-comodule $Ft$), so the two categories are not dual to each other.  \\
Note that this coalgebra has some other very nice ``finitary'' properties, and still the Gruson-Jensen duality does not
hold. We can see that the second term of the coradical filtration of $H$ is $H_1=F\{c_0,c_1,x_0,x_1,t\}$. Then,
since $\dim(H_1)<\infty$, by \cite[Theorem 2.8]{CNV} we get that $H$ is left and right almost noetherian, and
therefore, it is also left and right $\Ff$-noetherian. Moreover, $H$ is left artinian (i.e. artinian as a left
$H$-comodule) and $H^*$ is left noetherian (but
not right noetherian). This can be seen by looking at the dual ring $H^*=\left(\begin{array}{cc}C^* & C^* \\ 0 &
F\end{array} \right)$ of $H$. Then, by well known facts about matrix rings of this type, $H^*$ is left noetherian
because $C^*=F[[X]]$ and $F$ are noetherian and $C^*$ is left finitely generated, but it is not right noetherian because
$C^*$ is not finitely generated over $F$. The fact that $H$ is left artinian can be easily seen either because $H^*$ is
left noetherian, or because each indecomposable injective left $H$-comodule has only finite dimensional subcomodules.  
\end{example}

\begin{remark}
We note that a left and right semiperfect coalgebra is left and right $\Ff$-noetherian, 
and the left (and right) $\Ff$-noetherian is a categorical condition for a coalgebra: it means that every finite
dimensional  rational $C^*$-module is finitely presented as left $C^*$-module, that is, the categories of ${\rm
fp}(\Mm^C)$ and $f.d.{\rm fp}{}_{C^*}\Mm$ (finitely presented $C^*$-modules which are finite dimensional) coincide.
Because of this, one could then expect that the result on the G-J duality for $C^*$ might offer some insight to such a
duality between the functor rings of ${\rm fp}(\Mm^C)$ and ${\rm fp}({}^C\Mm)$, but the above example shows that this
does not hold. In fact, it is seen that even stronger conditions, such as $C^*$ almost noetherian on both sides and even
noetherian on one side are not enough to have the G-J duality. The above example is also motivated by the following
fact: a coalgebra $C$ which is right semiperfect and left artinian (i.e. $C$ is artinian as a left $C$-comodule,
equivalently, $C^*$ is left noetherian) is necessarily finite dimensional. Indeed, if $C$ is left artinian, then
$C=\bigoplus\limits_{i=1}^nE(S_i)$ a finite sum of artinian indecomposable injectives in ${}^C\Mm$; but since $C$ is
right semiperfect, these $E(S_i)$ are finite dimensional, and so $C$ is finite dimensional. Thus, the conditions of $C$
being left artinian and $C$ being right semiperfect can be thought as two ramifications of the finite dimensional
coalgebras whose ``intersection'' is the class of finite dimensional coalgebras. 
\end{remark}

In view of the above and of Example \ref{ex.6}, it is then natural to ask the following question:

\strut {\bf Question} {\it If $C$ is a left and right artinian coalgebra, are the categories ${}^C\Mm$ and $\Mm^C$
symmetric?} \strut 

This would in fact provide a first example of a Gruson-Jensen duality (symmetry) where the categories of comodules in
question are of quite a different nature than the categories of unitary modules over a ring with enough idempotents.
Examples of Gruson-Jensen duality for categories other than categories of unitary modules over a ring with enough
idempotents are scarce, apparently the only other known generic example of this type being the case of the categories of
unitary and torsionfree modules over an idempotent ring, provided they are locally finitely generated (see
\cite{CG}). Let us point out that the main obstacle in establishing a Gruson-Jensen duality for such categories is the
lack of enough projective objects.

\begin{center}
\sc Acknowledgment
\end{center}
The authors wish to thank the referee for careful comments, and for suggesting the second part of Theorem
\ref{t:semiperf}. We also wish to thank Stefaan Caenepeel for useful discussions on the subject and the hospitality of
Vrije Universiteit Brussels in the summer of 2009. The second author also wishes to thank the first for his hospitality
during his 2010 visit in the Department of Mathematics of ``Babe\c s-Bolyai'' University of Cluj-Napoca. 

The first author acknowledges the support of the Romanian grant PN-II-ID-PCE-2008-2 project ID\_2271. For the second
author, this work was supported by the strategic grant POSDRU/89/1.5/S/58852, Project ``Postdoctoral programe for
training scientific researchers'' cofinanced by the European Social Fund within the Sectorial Operational Program Human
Resources Development 2007-2013.

\end{document}